\documentclass[11pt, a4paper]{article}
\usepackage[T1]{fontenc}
\usepackage{amsfonts}
\usepackage{amsmath}
\usepackage{amssymb}
\usepackage{amsthm}
\usepackage{bbm}
\usepackage{bm}
\usepackage{mathrsfs}
\usepackage{verbatim}
\usepackage{setspace}
\usepackage{color}
\usepackage{pdfsync}
\usepackage{enumitem}

\theoremstyle{plain}
\newtheorem{theorem}{Theorem}[section]

\newtheorem{lemma}[theorem]{Lemma}

\theoremstyle{definition}

\newtheorem{remark}[theorem]{Remark}

\newtheorem{example}[theorem]{Example}

\newtheorem{assumption}[theorem]{Assumption}
\theoremstyle{remark}

\renewenvironment{thebibliography}[1]{%
\begin{oldthebibliography}{#1}%
\setlength{\baselineskip}{.9em}
\linespread{1}%
\small
\setlength{\parskip}{0ex}%
\setlength{\itemsep}{.2em}%
}%
{%
\end{oldthebibliography}%
}

\newcommand{\F}{\mathbb{F}}

\newcommand{\R}{\mathbb{R}}

\newcommand{\cE}{\mathcal{E}}
\newcommand{\cF}{\mathcal{F}}

\newcommand{\cN}{\mathcal{N}}

\newcommand{\cP}{\mathcal{P}}

\DeclareMathOperator{\Var}{Var}

\DeclareMathOperator{\limmed}{lim\, med}

\newcommand{\as}{\mbox{-a.s.}}

\newcommand{\1}{\mathbf{1}}

\newcommand{\br}[1]{\langle #1 \rangle}

\numberwithin{equation}{section}

\usepackage[pdfborder={0 0 0}]{hyperref}
\hypersetup{
  urlcolor = black,
  pdfauthor = {Marcel Nutz},
  pdfkeywords = {Pathwise stochastic integral, aggregation, non-dominated model, second order BSDE, $G$-expectation, medial limit},
  pdftitle = {Pathwise Construction of Stochastic Integrals},
  pdfsubject = {Pathwise Construction of Stochastic Integrals},
  pdfpagemode = UseNone
}

\begin{document}

\title{\vspace{-0.0cm}
Pathwise Construction of Stochastic Integrals
\date{First version: August 14, 2011. This version: June 12, 2012.}
\author{
  Marcel Nutz%
  \thanks{
  Dept.\ of Mathematics, Columbia University, New York, \texttt{mnutz@math.columbia.edu}. Financial support by European Research Council Grant 228053-FiRM is gratefully acknowledged.
  }
 }
}
\maketitle \vspace{-1em}

\begin{abstract}
We propose a method to construct the stochastic integral simultaneously under a non-dominated family of probability measures.
Path-by-path, and without referring to a probability measure, we construct a sequence of Lebesgue-Stieltjes integrals whose medial limit coincides with the usual stochastic integral under essentially any probability measure such that the integrator is a semimartingale. This method applies to any predictable integrand.
\end{abstract}

\vspace{.9em}

{\small

}

\section{Introduction}\label{se:intro}

The goal of this article is to construct the stochastic integral in a setting where a large family $\cP$ of probability measures is considered simultaneously. More precisely, given a predictable integrand $H$ and a process $X$ which is a semimartingale under all $P\in\cP$, we wish to construct a process  which $P$-a.s.\ coincides with the $P$-It\^o integral
${}^{(P)\hspace{-5pt}}\int H\,dX$ for all $P\in\cP$; i.e., we seek to \emph{aggregate} the family $\{{}^{(P)\hspace{-5pt}}\int H\,dX\}_{P\in\cP}$ into a single process.
This work is motivated by recent developments in probability theory and stochastic optimal control, where stochastic integrals under families of measures have arisen in the context of Denis and Martini's model of volatility uncertainty in financial markets~\cite{DenisMartini.06}, Peng's $G$-expectation~\cite{Peng.10} and in particular the representation of $G$-martingales of Soner et al.~\cite{SonerTouziZhang.2010rep}, the second order backward stochastic differential equations and target problems of Soner et al.~\cite{SonerTouziZhang.2010bsde, SonerTouziZhang.2010dual} and in the non-dominated optional decomposition of Nutz and Soner~\cite{NutzSoner.10}. In all these examples, one considers a family $\cP$ of (mutually singular) measures which cannot be dominated by a finite measure.

The key problem in stochastic integration is, of course, that the paths of the integrator $X$ are not of finite variation. The classical Hilbert space construction depends strongly on the probability measure since it exploits the martingale properties of the integrator; in particular, it is far from being pathwise. A natural extension of the classical construction to a family $\cP$, carried out in \cite{DenisMartini.06, Peng.10, LiPeng.11} under specific regularity assumptions, is to consider the upper expectation $\cE_\cP[\,\cdot\,]=\sup_{P\in \cP} E^P[\,\cdot\,]$ instead of the usual expectation and define the stochastic integral along the lines of the usual closure operation from simple integrands, but under a norm induced by $\cE_\cP$. Since such a norm is rather strong, this leads to a space of integrands which is smaller than in the classical case.

A quite different approach is to construct the stochastic integral pathwise and without direct reference to a probability measure; in this case, the integral will be well defined under all $P\in\cP$. The strongest previous result about pathwise integration is due to Bichteler~\cite{Bichteler.81}; it includes in particular the integrals which can be defined by F\"ollmer's approach~\cite{Foellmer.81}. A very readable account of that result and some applications also appear in Karandikar~\cite{Karandikar.95}. Bichteler's remarkable observation is that if $H=G_-$ is the left limit of a c\`adl\`ag process $G$, then one can obtain a very favorable discretization of  $\int H\,dX$ by sampling $H$ at a specific sequence of stopping times, given by the level-crossing times of $H$ at a grid of mesh size $2^{-n}$. Namely, the corresponding Riemann sums $\int H^n\,dX$ converge \emph{pointwise} in $\omega$ (uniformly in time, outside a set which is negligible for all $P$), and this limit yields a pathwise definition of the integral. To appreciate this fact, recall that as soon as $H$ is left-continuous, essentially any discretization will converge to the stochastic integral, but only \emph{in measure}, so that it is not immediate to construct a single limiting process for all $P$: passing to a subsequence yields $P$-a.s.\ convergence for some $P$, but not for all $P$ at once.

Finally, let us mention that the ``skeleton approach'' of Willinger and Taqqu~\cite{WillingerTaqqu.88,WillingerTaqqu.89} is yet another pathwise stochastic integration theory (related to the problem of martingale representation); however, this construction cannot be used in our context since it depends strongly on the equivalence class of the probability measure.
The main drawback of the previous results is that the restrictions on the admissible integrands can be too strong for applications. In particular, the integrals appearing in the main results of \cite{SonerTouziZhang.2010bsde} and~\cite{NutzSoner.10} could not be aggregated for that reason. In a specific setting where $X$ is continuous, one possible solution, proposed by Soner et al.~\cite{SonerTouziZhang.2010aggreg}, is to impose a strong separability assumption on the set $\cP$, which allows to glue together directly the processes $\{{}^{(P)\hspace{-5pt}}\int H\,dX\}_{P\in\cP}$.

In the present paper, we propose a surprisingly simple, pathwise construction of the stochastic integral for arbitrary predictable integrands $H$ (Theorem~\ref{th:aggregInt}) and a very general set $\cP$. In a first step, we average $H$ in the time variable to obtain approximations $H^n$ of finite variation, which allows to define the integral $\int H^n\,dX$ pathwise. This averaging requires a certain domination assumption; however, by imposing the latter at the level of predictable characteristics, we achieve a condition which is satisfied in all cases of practical interest (Assumption~\ref{ass:domination}). The second step is the passage to the limit, where we shall work with the (projective limit of the) convergence in measure
and use a beautiful construction due to G.~Mokobodzki, known as medial limit (cf.\ Meyer~\cite{Meyer.73} and Section~\ref{se:aggreg} below). As a Banach limit, this is not a limit in a proper sense, but it will allow us to define path-by-path a measurable process $\int H\,dX$ which coincides with the usual stochastic integral under every $P$; in fact,  it seems that this technology may be useful in other aggregation problems as well. To be precise, this requires a suitable choice of the model of set theory: we shall work under the Zermelo--Fraenkel set theory with axiom of choice (ZFC) plus the Continuum Hypothesis.

\section{Main Result}

Let $(\Omega,\cF)$ be a measurable space equipped with a right-continuous filtration $\F=(\cF_t)_{t\in[0,1]}$ and let $\cP$ be a family of probability measures on $(\Omega,\cF)$. In the sequel, we shall work in the $\cP$-universally augmented filtration
\[
  \F^*=(\cF^*_t)_{t\in[0,1]},\quad \cF^*_t:=\bigcap_{P\in\cP} \cF_t\vee \cN^P,
\]
where $\cN^P$ is the collection of $(\cF,P)$-nullsets (but see also Remark~\ref{rk:concluding}(ii)). Moreover, let $X$ be an adapted process with c\`adl\`ag paths such that $X$ is a semimartingale under each $P\in\cP$, and let $H$ be a predictable process which is $X$-integrable under each $P\in\cP$.

Since we shall average $H$ in time, it is necessary to fix a measure on $[0,1]$, at least path-by-path. We shall work under the following condition; see Jacod and Shiryaev~\cite[Section~II]{JacodShiryaev.03} for the notion of predictable characteristics.

\begin{assumption}\label{ass:domination}
  There exists a predictable c\`adl\`ag increasing process $A$ such that
  \[
    \Var(B^P) + \br{X^c}^P + (x^2\wedge 1) * \nu^P \ll A\quad P\as\quad\mbox{for all}\quad P\in\cP;
  \]
  where $(B^P, \br{X^c}^P, \nu^P)$ is the triplet of predictable characteristics of $X$ under $P$ and $\Var(B^P)$ denotes the total variation of $B^P$.
\end{assumption}

With ${}^{(P)\hspace{-5pt}}\int H\,dX$ denoting the It\^o integral under $P$, our main result can be stated as follows.

\begin{theorem}\label{th:aggregInt}
  Under Assumption~\ref{ass:domination}, there exists an $\F^*$-adapted c\`adl\`ag process, denoted by $\int H\,dX$, such that
  \[
    \int H\,dX=\sideset{^{(P)\hspace{-7pt}}}{}{\int}H\,dX\quad P\as\quad\mbox{for all}\quad P\in\cP.
  \]
  Moreover, the construction of any path $(\int H\,dX)(\omega)$ involves only the paths $H(\omega)$ and $X(\omega)$.
\end{theorem}

\pagebreak

Assumption~\ref{ass:domination} is quite weak and should not be confused with a domination property for $\cP$ or the paths of $X$. In fact, most semimartingales of practical interest have characteristics absolutely continuous with respect to $A_t=t$ (diffusion processes, solutions of L\'evy driven stochastic differential equations, etc.). The following example covers the applications from the introduction.

\begin{example}
  Let $X$ be a continuous local martingale under each $P\in\cP$, then Assumption~\ref{ass:domination} is satisfied. Indeed, let
  \[
    A:= X^2-X_0^2-2\int X\,dX;
  \]
  here the stochastic integral can be defined pathwise by Bichteler's construction~\cite[Theorem~7.14]{Bichteler.81}. Then $A$ is a continuous process and, by It\^o's formula,  $A=[X]-X_0^2=\br{X^c}^P$ $P$-a.s.\ for every $P\in\cP$. Therefore, Assumption~\ref{ass:domination} is satisfied with equality.
\end{example}

The previous example should illustrate that Assumption~\ref{ass:domination} is much weaker than it may seem at first glance. For instance, let $X$ be the canonical process on $\Omega=C([0,1];\R)$ and let $\Lambda$ be the set of \emph{all} increasing continuous functions $f: [0,1]\to \R_+$ with $f(0)=0$. Using time-changed Brownian motions, construct a situation where for any $f\in\Lambda$ there exists $P\in\cP$ under which $f$ is the quadratic variation of $X$, $P$-a.s.
Then we observe that, by the above, Assumption~\ref{ass:domination} is satisfied---even though it is clearly impossible to dominate the set $\Lambda$. The crucial point here is the flexibility to assign different values to $A$ on the various supports of the measures $P$.

\subsection{Approximating Sequence}\label{se:approx}

Since our aim is to prove Theorem~\ref{th:aggregInt}, we may assume without loss of generality that $X_0=0$. Moreover, we may assume that the jumps of $X$ are bounded by one in magnitude,
\[
 |\Delta X|\leq 1.
\]
Indeed, the process $\check X:=\sum_{s\leq \cdot}  \1_{\{|\Delta X_s|>1\}} \Delta X_s$ is of finite variation and $\int H\,d\check X$ is easily defined since it is simply a sum. Decomposing
\[
  X=(X-\check X) +\check X,
\]
it suffices to construct the integral $\int H\,d(X-\check X)$ whose integrator has jumps bounded by one; moreover, $X-\check X$ satisfies Assumption~\ref{ass:domination} if $X$ does. (Of course, we cannot reduce further to the martingale case, since the semimartingale decomposition of $X$ depends on $P$!)

In this section, we construct an approximating sequence of integrands $H^n$ such that the integrals $\int H^n\,dX$ can be defined pathwise and tend to the integral of $H$ in measure. To this end, we shall assume that $H$ is uniformly bounded by a constant,
\[
  |H|\leq c.
\]
In fact, we can easily remove this condition later on, since
\begin{equation}\label{eq:truncationConv}
  \sideset{^{(P)\hspace{-7pt}}}{}{\int} H\1_{\{|H|\leq n\}}\,dX \to \sideset{^{(P)\hspace{-7pt}}}{}{\int} H \,dX\quad \mbox{in}\quad ucp(P)\quad\mbox{for all}\quad P\in\cP
\end{equation}
by the definition of the usual stochastic integral. Here $ucp(P)$ stands for convergence in probability $P$, uniformly (on compacts) in time.
We recall the process $A$ from Assumption~\ref{ass:domination}.
We may assume that $A_t-A_s\geq t-s$ for all $0\leq s\leq t\leq 1$ by replacing $A_t$ with $A_t+t$ if necessary; moreover, to avoid complicated notation, we define $H_t=A_t=0$ for $t<0$.

\begin{lemma}\label{le:approxSeq}
  For each $n\geq 1$, define the Lebesgue-Stieltjes integral
  \begin{equation}\label{eq:Hn}
    H^n_t:=\frac{1}{A_t-A_{t-1/n}}\int_{t-1/n}^t H_s\,dA_s,\quad t>0
  \end{equation}
  and $H^n_0:=0$. Then
  \begin{equation}\label{eq:Yn}
    Y^n:=H^nX- \int X_-\, dH^n
  \end{equation}
  is well defined in the Lebesgue-Stieltjes sense and satisfies
  \begin{equation*}%
    Y^n= \sideset{^{(P)\hspace{-7pt}}}{}{\int} H^n\,dX \to \sideset{^{(P)\hspace{-7pt}}}{}{\int} H\,dX\quad \mbox{in}\quad ucp(P)\quad\mbox{for all}\quad P\in\cP.
  \end{equation*}
\end{lemma}

\begin{proof}
  Recalling that $|H|\leq c$ and that $A$ is a predictable increasing c\`adl\`ag process, we see that $H^n$ is a predictable process satisfying $|H^n|\leq c$ identically and having c\`adl\`ag path of finite variation. In particular, we can use the Lebesgue-Stieltjes integral to define the process $Y^n$ pathwise via~\eqref{eq:Yn}.
  We deduce via integration by parts
  that $Y^n$ coincides $P$-a.s.\ with the stochastic integral ${}^{(P)\hspace{-5pt}}\int H^n\,dX$ for each $P\in\cP$.

  By the standard theorem on approximate identities, we have
  \begin{equation}\label{eq:HnToH}
    H^n (\omega) \to H(\omega)\quad\mbox{in}\quad L^1([0,1],dA(\omega))\quad\mbox{for all}\quad \omega\in\Omega.
  \end{equation}
  For the remainder of the proof, we fix $P\in\cP$ and show the convergence of $Y^n$ to ${}^{(P)\hspace{-5pt}}\int H\,dX$ in $ucp(P)$. Since $P$ is fixed, we may use the usual tools of stochastic analysis under $P$ and write, as usual, $E$ for the expectation operator under $P$, etc. (One can pass to the augmentation of $\F^*$ under $P$ to have the ``usual assumptions'', although this is not important here.)

  Recall that the jumps of $X$ are bounded, so that there is a canonical decomposition
  $X=M+B$, where $M$ is a local martingale and $B$ is predictable of finite variation. Since the jumps of $M$ and $B$ are then also bounded, a standard localization argument allows us to assume that $\Var(B)$ and the quadratic variation $[M]$ are uniformly bounded. We have
  \begin{align*}
    E\bigg[\sup_{t\leq 1}& \bigg|\int_0^t H^n \,dX - \int_0^t H \,dX\bigg|^2\bigg] \\
    & \leq 2 E\bigg[\sup_{t\leq 1} \bigg|\int_0^t (H^n -H) \,dM\bigg|^2\bigg] + 2 E \bigg[\bigg|\int_0^1 |H^n -H| \,d\Var(B)\bigg|^2\bigg].
  \end{align*}
  The second expectation on the right hand side converges to zero; indeed, recalling that $|H|,|H^n|\leq c$, we see that $\int_0^1 |H^n -H| \,d\Var(B)$ is uniformly bounded and converges to zero pointwise. The latter follows from~\eqref{eq:HnToH} since $\Var(B)(\omega)\ll A(\omega)$ and
  \[
    \bigg\{|H^n(\omega)-H(\omega)|\frac{d\Var(B)(\omega)}{dA(\omega)}\bigg\}_{n\geq1}\subseteq L^1([0,1],dA(\omega))
  \]
  is uniformly integrable.

  It remains to show that the first expectation converges to zero. Let $\br{M}$ be the predictable compensator of $[M]$, then the Burkholder-Davis-Gundy inequalities yield
  \begin{align*}
    E\bigg[\sup_{t\leq 1} \bigg|\int_0^t (H^n -H) \,dM\bigg|^2\bigg]
    & \leq 4 E\bigg[\int_0^1 |H^n -H|^2 \,d[M]\bigg] \\
    & = 4 E\bigg[\int_0^1 |H^n -H|^2 \,d\br{M}\bigg].
  \end{align*}
  Recalling that $|\Delta X|\leq 1$, we have that
  \[
    \br{M} = \br{X^c} + (x^2\wedge 1) * \nu - \sum_{s\leq\cdot } (\Delta B_s)^2
  \]
  and in particular that $\br{M}\ll A$. Since $\br{M}$ is bounded like $[M]$, we conclude exactly as above that
  $E[\int_0^1 |H^n -H|^2 \,d\br{M}]$ converges to zero.
\end{proof}

\subsection{Aggregation by Approximation in Measure}\label{se:aggreg}

In this section, we shall find it very useful to employ Mokobodzki's medial limit, which yields a universal method (i.e., independent of the underlying probability) to identify the limit of a sequence which converges in probability. More precisely,
``$\limmed$'' is a mapping on the set of real sequences with the following property (cf.\ \cite[Theorems~3, 4]{Meyer.73}): If $(Z_n)$ is a sequence of random variables on a measurable space $(\Omega',\cF')$, then $Z(\omega):=\limmed_n Z_n(\omega)$ is universally measurable and if $P$ is a probability measure on $(\Omega',\cF')$ such that $Z_n$ converges to some random variable $Z^P$ in probability $P$, then $Z=Z^P$ $P$-a.s. Here uniform measurability refers to the universal completion of $\cF'$ under all probability measures on ($\Omega',\cF'$).

Although developed in a different context and apparently not used before in ours, medial limits seem to be tailored to our task. Their construction is usually achieved through
a transfinite induction that uses the Continuum Hypothesis (cf.\ \cite{Meyer.73}); in fact, it is known that medial limits exist under weaker hypotheses (Fremlin~\cite[538S]{Fremlin.08}), but not under ZFC alone (Larson~\cite{Larson.09}). We shall adopt a sufficient set of axioms; since the Continuum Hypothesis is independent of ZFC, we consider this a pragmatic choice of the model of set theory for our purposes.

We have the following result for c\`adl\`ag processes.

\begin{lemma}\label{le:limmed}
  Let $(Y^n)_{n\geq1}$ be a sequence of $\F^*$-adapted c\`adl\`ag processes. Assume that for each $P\in\cP$ there exists a c\`adl\`ag process
  $Y^P$ such that $Y^n_t\to Y^P_t$ in measure $P$ for all $t\in[0,1]$. Then there exists an $\F^*$-adapted c\`adl\`ag process $Y$ such that $Y=Y^P$ $P$-a.s.\ for all $P\in\cP$.
\end{lemma}

\begin{proof}
   Let $r\in[0,1]$ be a rational number and define $\tilde{Y}_r:=\limmed_n Y^n_r$. Then $\tilde{Y}_r$ is measurable with respect to the universal completion of $\cF_r$, which is contained in $\cF_r^*$. Moreover,
  \begin{equation}\label{eq:proofLimmedRat}
    \tilde{Y}_r=Y_r^P\quad P\as\quad\mbox{for all}\quad P\in\cP.
  \end{equation}
  Given arbitrary $t\in[0,1)$, we define $Y_t:=\limsup_{r\downarrow t} \tilde{Y}_r$, where $r$ is rational (and $Y_1:=\tilde{Y}_1$). Fix $P\in\cP$. Since $Y^P$ is c\`adl\`ag,~\eqref{eq:proofLimmedRat} entails that
  \[
    Y_t= \limsup_{r\downarrow t} \tilde{Y}_r = \limsup_{r\downarrow t} Y^P_r = Y^P_t\quad P\as
  \]
  and that the $\limsup$ is actually a limit outside a $P$-nullset. As a consequence,
  \[
    N=\{\omega\in\Omega:\, Y_\cdot(\omega) \mbox{ is not c\`adl\`ag}\}
  \]
  is a $P$-nullset. Since $P\in\cP$ was arbitrary, $N$ is actually a nullset under every $P\in\cP$ and therefore contained in $\cF_0^*$. We redefine $Y\equiv 0$ on $N$, then $Y$ is $\F^*$-adapted and all paths of $Y$ are c\`adl\`ag. Since $Y$ is also a $P$-modification of $Y^P$, we have $Y=Y^P$ $P$-a.s.
\end{proof}

Our main result can then be proved as follows.

\begin{proof}[Proof of Theorem~\ref{th:aggregInt}.]
  We first assume that $H$ is uniformly bounded. Then Lemma~\ref{le:approxSeq} yields a sequence $\int H^n\,dX$ of pathwise defined integrals which converge in $ucp(P)$ for all $P\in\cP$. According to Lemma~\ref{le:limmed}, there exists a process $\int H\,dX$ with the desired properties. For general $H$, we use the previous argument to define $\int H\1_{|H|\leq n}\,dX$ for $n\geq1$. In view of~\eqref{eq:truncationConv}, we may apply Lemma~\ref{le:limmed} once more to obtain $\int H\,dX$.
\end{proof}

\pagebreak

\begin{remark}\label{rk:concluding}
  (i) If the integrand $H$ has left-continuous paths, the assertion of Theorem~\ref{th:aggregInt} holds true without Assumption~\ref{ass:domination}.
  Indeed, set $A_t:=t$ and define $H^n$ as in~\eqref{eq:Hn}. Then, by the left-continuity, we have $H^n_t(\omega)\to H_t(\omega)$ for all $t$ and $\omega$, without exceptional set. The rest of the proof is as above. (Of course, there are other ways to define $H^n$ in this case, such as discretization.)

  (ii) Theorem~\ref{th:aggregInt} can be obtained in a filtration slightly smaller than $\F^*$. Indeed, the same proofs apply if $\F^*$ is replaced by the filtration obtained as follows: first, augment $\F$ by the collection $\cap_{P\in\cP} \cN^P$ of $\cP$-polar sets, then, take the universal augmentation with respect to all probability measures (and not just those in $\cP$). Our proofs also show that the random variable $\int_0^1 H\,dX$ is measurable with respect to the universal completion of $\cF_1$, without adding the $\cP$-polar sets.

  (iii) Needless to say, our ``construction'' of the stochastic integral is not ``constructive'' in the proper sense; it merely yields an existence result.
\end{remark}

\newcommand{\dummy}[1]{}

\end{document}